\documentclass[12pt]{article}

\usepackage{setspace}
\usepackage{amsfonts}
\usepackage{amsmath}
\usepackage{amssymb}
\usepackage[utf8]{inputenc}
\usepackage[T1]{fontenc}
\usepackage {eurosym}                                
\usepackage {color}                                  
\usepackage {pifont}                                 
\usepackage{setspace}
\usepackage {multicol}                               
\usepackage {times}
\usepackage{latexsym}
\usepackage{tabularx}
\usepackage{verbatim}

\usepackage{makeidx}
\usepackage{color}
\usepackage [scaled=.90]{helvet}                     
\usepackage {courier}                                
\usepackage {graphicx}                               
\usepackage {array}                                  
\usepackage {longtable}                              
\makeindex
\usepackage{fancyvrb}

\newtheorem{theorem}{Theorem}
\newtheorem{algorithm}{Algorithm}
\newtheorem{lemma}{Lemma}
\newtheorem{coro}{Corollary}
\newtheorem{defi}{Definition}
\newtheorem{remark}{Remark}

\newcommand{\NN}{{\mathbb N}}

\newcommand{\ZZ}{{\mathbb Z}}

\newcommand{\FF}{{\mathbb F}}
\newcommand{\bsalpha}{{\boldsymbol \alpha}}
\newcommand{\bsh}{{\boldsymbol h}}

\newcommand{\bk}{{\bf k}}
\newcommand{\bb}{{\bf b}}
\newcommand{\bc}{{\bf c}}
\newcommand{\bh}{{\bf h}}
\newcommand{\by}{{\bf y}}
\newcommand{\bszero}{{\boldsymbol 0}}
\newcommand{\bigO}{O}
\newenvironment{proof}{\begin{trivlist}
    \item[\hskip\labelsep{\it Proof.}]}{$\hfill\Box$\end{trivlist}}
\title{Vandermonde Nets}

\author{Roswitha Hofer and Harald Niederreiter}

\date{}

\begin{document}

\setcounter{page}{1}
\maketitle

\begin{abstract}
The second author recently suggested to identify the generating matrices of a digital $(t,m,s)$-net over the finite field $\FF_q$ with an $s \times m$ 
matrix $C$ over $\FF_{q^m}$. More exactly, the entries of $C$ are determined by interpreting the rows of the generating matrices as elements of 
$\FF_{q^m}$. This paper introduces so-called Vandermonde nets, which correspond to Vandermonde-type matrices $C$, and discusses the quality parameter 
and the discrepancy of such nets. The methods that have been successfully used for the investigation of polynomial lattice point sets and hyperplane nets 
are applied to this new class of digital nets. In this way, existence results for small quality parameters and good discrepancy bounds are obtained. 
Furthermore, a first step towards component-by-component constructions is made. A novelty of this new class of nets is that explicit constructions of 
Vandermonde nets over $\FF_q$ in dimensions $s\leq q+1$ with best possible quality parameter can be given. So far, good explicit constructions of the competing 
polynomial lattice point sets are known only in dimensions $s\leq 2$. 
\end{abstract}



\section{Introduction and basic definitions} \label{sein}

In this paper, we introduce a new family of digital nets over finite fields. A net, or more precisely a $(t,m,s)$-net, is a finite collection of
points (also called a point set) in the $s$-dimensional half-open unit cube $[0,1)^s$ possessing equidistribution properties. A digital net is a
net obtained by the linear-algebra construction described below. Various constructions of nets are already known, and most of them are digital nets.
Reviews of the theory of nets can be found in the monograph~\cite{DP10} and in the recent survey article~\cite{N13}.
 
Let $\FF_q$ be the finite field of order $q$, where $q$ is an arbitrary prime power, and let $m$ and $s$ be positive integers. In order to construct
a digital $(t,m,s)$-net over $\FF_q$, we choose $m \times m$ matrices $C^{(1)},\ldots,C^{(s)}$ over $\FF_q$, called the \emph{generating matrices}
of the digital net. We write $\ZZ_q = \{0,1,\ldots,q-1\} \subset \ZZ$ for the set of digits in base $q$. We define the map $\Psi_m :
\FF_q^m \to [0,1)$ by
$$
\Psi_m(\bh^{\top})=\sum_{j=1}^m \psi(h_j) q^{-j}
$$
for any column vector $\bh^{\top}=(h_1,\ldots,h_m)^{\top} \in \FF_q^m$, where $\psi: \FF_q \to \ZZ_q$ is a chosen bijection. With a fixed column vector
$\bb^{\top} \in \FF_q^m$, we associate the point
\begin{equation} \label{eqps}
\left(\Psi_m(C^{(1)} \bb^{\top}),\ldots,\Psi_m(C^{(s)} \bb^{\top}) \right) \in [0,1)^s.
\end{equation}
By letting $\bb^{\top}$ range over all $q^m$ column vectors in $\FF_q^m$, we arrive at a point set consisting of $q^m$ points in $[0,1)^s$. This construction
of digital nets can be generalized somewhat by employing further bijections between $\FF_q$ and $\ZZ_q$ (see \cite[p.~63]{N92}), but this is not needed for
our purposes since our results depend just on the generating matrices. For $i=1,\ldots,s$ and $j=1,\ldots,m$, let $\bc_j^{(i)} \in \FF_q^m$ denote the $j$th
row vector of the matrix $C^{(i)}$.

\begin{defi} \label{dedn} {\rm
Let $q$ be a prime power and let $t$, $m$, and $s$ be integers with $0 \le t \le m$, $m \ge 1$, and $s \ge 1$. Then the point set consisting of the 
$q^m$ points in~\eqref{eqps} is a
\emph{digital $(t,m,s)$-net over} $\FF_q$ if for any nonnegative integers $d_1,\ldots,d_s$ with $\sum_{i=1}^s d_i = m-t$, the $m-t$ vectors $\bc_j^{(i)} \in
\FF_q^m$ with $1 \le j \le d_i$ and $1 \le i \le s$ are linearly independent over $\FF_q$ (the empty collection of vectors occurring in the case $t=m$ is
considered linearly independent over $\FF_q$).}
\end{defi}

It is evident that the condition in Definition~\ref{dedn} becomes the stronger the smaller the value of $t$. The main interest is therefore in constructing
digital $(t,m,s)$-nets over $\FF_q$ with a small value of $t$. The number $t$ is called the \emph{quality parameter} of a digital $(t,m,s)$-net over $\FF_q$.

\begin{remark} \label{redn} {\rm
The definition of a digital $(t,m,s)$-net over $\FF_q$ can be translated into an explicit equidistribution property of the points of the digital net as follows.
Consider any subinterval $J$ of $[0,1)^s$ of the form
$$
J=\prod_{i=1}^s [a_iq^{-d_i},(a_i+1)q^{-d_i})
$$
with $a_i,d_i \in \ZZ$, $d_i \ge 0$, $0 \le a_i < q^{d_i}$ for $1 \le i \le s$, and with $J$ having $s$-dimensional volume $q^{t-m}$. Then any such interval $J$ contains
exactly $q^t$ points of the digital net. The proof of this fact can be found, for instance, in \cite[Section 4.4.2]{DP10}. From this point of view, it is again
clear that we are interested in small values of $t$ since then the family of intervals $J$ for which the above equidistribution property holds becomes larger.}
\end{remark}

Our starting point for the construction of new digital nets is the suggestion made in \cite[Remark~6.3]{N13} to view the row vectors of the generating matrices
as elements of the finite field $\FF_{q^m}$ (which is isomorphic to $\FF_q^m$ as an $\FF_q$-linear space). Thus, we consider elements $\gamma_j^{(i)} \in
\FF_{q^m}$, $1 \le i \le s$, $1 \le j \le m$, and the $j$th row of $C^{(i)}$ is then obtained as $\bc_j^{(i)} = \phi(\gamma_j^{(i)})$, where $\phi : \FF_{q^m}
\to \FF_q^m$ is a fixed vector space isomorphism (or, equivalently, $\bc_j^{(i)}$ is the coordinate vector of $\gamma_j^{(i)}$ relative to a fixed ordered basis
of $\FF_{q^m}$ over $\FF_q$). Again following \cite[Remark~6.3]{N13}, we arrange the $\gamma_j^{(i)}$ into an $s \times m$ matrix 
$C=(\gamma_j^{(i)})_{1 \le i \le s, \; 1 \le j \le m}$ over $\FF_{q^m}$, and we have then a single matrix that governs the construction of the digital net.
Because of the vector space isomorphism between $\FF_{q^m}$ and $\FF_q^m$, the following observation is an immediate consequence of Definition~\ref{dedn}.

\begin{lemma} \label{leed}
The digital net obtained from the matrix $C=(\gamma_j^{(i)})_{1 \le i \le s, \; 1 \le j \le m}$ over $\FF_{q^m}$ is a digital $(t,m,s)$-net over $\FF_q$ if and
only if, for any integers $d_1,\ldots,d_s \ge 0$ with $\sum_{i=1}^s d_i = m-t$, the $m-t$ elements $\gamma_j^{(i)} \in \FF_{q^m}$ with $1 \le j \le d_i$ and $1 \le i \le s$ 
are linearly independent over $\FF_q$.
\end{lemma}

It is of apparent interest to consider a matrix $C$ that is structured. In this paper, we analyze what happens when we choose a matrix $C$ that has a
Vandermonde-type structure. Concretely, we choose an $s$-tuple $\bsalpha =(\alpha_1,\ldots,\alpha_s) \in \FF_{q^m}^s$ and then we set up the $s \times m$ matrix $C=
(\gamma_j^{(i)})_{1 \le i \le s, \; 1 \le j \le m}$ over $\FF_{q^m}$ defined by $\gamma_j^{(1)}=\alpha_1^{j-1}$ for $1 \le j \le m$ and (if $s \ge 2$) 
$\gamma_j^{(i)}=\alpha_i^j$ for $2 \le i \le s$ and $1 \le j \le m$. We use the standard convention $0^0=1 \in \FF_q$. For obvious reasons, we call the digital
net obtained from $C$ a \emph{Vandermonde net over} $\FF_q$.

\begin{remark} \label{rej} {\rm
If $s \ge 2$, then for $2 \le i \le s$ we do not want to put $\gamma_j^{(i)}=\alpha_i^{j-1}$ for $1 \le j \le m$, since otherwise the elements $\gamma_1^{(1)}
=1 \in \FF_q$ and $\gamma_1^{(2)}=1 \in \FF_q$ are linearly dependent over $\FF_q$, and so the least value of $t$ such that the resulting digital net is a
digital $(t,m,s)$-net over $\FF_q$ is $t=m-1$.}
\end{remark}

\begin{remark} \label{rehy} {\rm
A broad class of digital nets, namely that of hyperplane nets, was introduced in~\cite{PDP} (see also \cite[Chapter~11]{DP10}). Choose $\alpha_1,\ldots,\alpha_s
\in \FF_{q^m}$ not all $0$. Then for the corresponding hyperplane net relative to a fixed ordered basis $\omega_1,\ldots,\omega_m$ of $\FF_{q^m}$ over $\FF_q$,
the matrix $C=(\gamma_j^{(i)})_{1 \le i \le s, \; 1 \le j \le m}$ in Lemma~\ref{leed} is given by $\gamma_j^{(i)}=\alpha_i \omega_j$ for $1 \le i \le s$ and
$1 \le j \le m$ (see \cite[Theorem~11.5]{DP10} and \cite[Remark~6.4]{N13}). Thus, $C$ is also a structured matrix, but the structure is in general not a
Vandermonde structure. Consequently, Vandermonde nets are in general not hyperplane nets relative to a fixed ordered basis of $\FF_{q^m}$ over $\FF_q$.}
\end{remark}

In this paper, we discuss various aspects of Vandermonde nets. Section~\ref{seqp} ensures the existence of Vandermonde nets having small quality parameter 
and as a by-product the existence of such nets satisfying good discrepancy bounds. This by-product is improved in Section~\ref{sec:disc} by using averaging 
arguments.  Section~\ref{seex} presents an explicit construction of Vandermonde nets over $\FF_q$ in dimensions $s\leq q+1$ with best possible quality parameter. 
Finally, Section~\ref{secb} breaks the first ground for component-by-component constructions of Vandermonde nets.

\section{Existence results for a small quality parameter} \label{seqp}

For the investigation of the quality parameter of a Vandermonde net over $\FF_q$, we make use of the following notation and conventions. We write $\FF_q[x]$
for the ring of polynomials over $\FF_q$ in the indeterminate $x$. For any integer $m \ge 1$, we put 
\begin{eqnarray*}
H_{q,m} & := & \{h\in\FF_q[x]:\deg(h)\leq m, \, h(0)=0\},\\
H_{q,m}^* & := & \{h \in \FF_q[x]: \deg(h) < m\},
\end{eqnarray*}
where $\deg(0):=0$. Furthermore, we define $\deg^*(h) := \deg(h)$ for $h \in \FF_q[x]$ with $h \ne 0$ and $\deg^*(0) := -1$. We write 
$\bsh:=(h_1,\ldots,h_s)\in\FF_q[x]^s$ for a given dimension $s \ge 1$. Finally, for any $\bsalpha =(\alpha_1,\ldots,\alpha_s) \in \FF_{q^m}^s$, we put
$$
D_{q,m,\bsalpha}:=\{\bsh\in H_{q,m}^* \times H_{q,m}^{s-1}:\sum_{i=1}^s h_i(\alpha_i)=0\}
$$
and $D'_{q,m,\bsalpha}:=D_{q,m,\bsalpha}\setminus\{\bszero\}$. 

We define the following figure of merit. We use the standard convention that an empty sum is equal to $0$. 

\begin{defi} \label{defm} {\rm
If $D'_{q,m,\bsalpha}$ is nonempty, we define the \emph{figure of merit}
$$
\varrho(\bsalpha):=\min_{\bsh\in D'_{q,m,\bsalpha}} \big(\deg^*(h_1)+\sum_{i=2}^s\deg(h_i) \big).
$$
Otherwise, we define $\varrho(\bsalpha):=m$.}
\end{defi}

It is trivial that we always have $\varrho(\bsalpha) \ge 0$. For $s=1$ it is clear that $\varrho(\bsalpha) \le m$. For $s \ge 2$ the $m+1$ elements $1,
\alpha_1,\ldots,\alpha_1^{m-1},\alpha_2 \in \FF_{q^m}$ are linearly dependent over $\FF_q$, and so again $\varrho(\bsalpha) \le m$. 

\begin{theorem} \label{thqp} 
Let $q$ be a prime power, let $s,m\in\NN$, and let $\bsalpha=(\alpha_1,\ldots,\alpha_s)\in\FF^s_{q^m}$. 
Then the Vandermonde net determined by $\bsalpha\in\FF^s_{q^m}$ is a digital $(t,m,s)$-net over $\FF_q$ with $t=m-\varrho(\bsalpha)$. 
\end{theorem}

\begin{proof}
The case $\varrho(\bsalpha)=0$ is trivial by the parenthetical remark in Definition \ref{dedn}, and so we can assume that $\varrho(\bsalpha)\geq 1$. 
In view of Lemma~\ref{leed}, it suffices to show that for any integers $d_1,\ldots,d_s \ge 0$ with $\sum_{i=1}^s d_i=\varrho(\bsalpha)$, the elements
$\alpha_1^j$ for $0 \le j \le d_1-1$ and $\alpha_i^j$ for $1 \le j \le d_i$, $2 \le i \le s$, are linearly independent over $\FF_q$. A purported nontrivial
linear dependence relation for these elements can be written in the form
$$
\sum_{i=1}^s h_i(\alpha_i)=0,
$$
with a nonzero $s$-tuple $\bsh =(h_1,\ldots,h_s) \in H_{q,m}^* \times H_{q,m}^{s-1}$ satisfying $\deg^*(h_1) < d_1$ and $\deg(h_i) \le d_i$ for $2 \le i \le s$.
It follows that
$$
\deg^*(h_1) + \sum_{i=2}^s \deg(h_i) < \sum_{i=1}^s d_i = \varrho(\bsalpha).
$$
But since $\bsh \in D'_{q,m,\bsalpha}$, this is a contradiction to the definition of $\varrho(\bsalpha)$.
\end{proof} 

\begin{remark} \label{repl} {\rm
It is of interest to compare Vandermonde nets with the polynomial lattice point sets introduced in~\cite{N92b} (see also \cite[Chapter~10]{DP10}, \cite[Section~4.4]{N92},
and the recent survey article~\cite{P12} for the theory of polynomial lattice point sets). We consider polynomial lattice point sets with a modulus $f \in 
\FF_q[x]$ which is irreducible over $\FF_q$ of degree $m$. An $s$-dimensional polynomial lattice point set depends also on the choice of polynomials
$g_1,\ldots,g_s \in H_{q,m}^*$. One arrives at a digital $(t,m,s)$-net over $\FF_q$ with a quality parameter $t$ depending on a figure of merit analogous to
$\varrho(\bsalpha)$ in Definition~\ref{defm}. The crucial condition $\sum_{i=1}^s h_i(\alpha_i)=0$ in the definition of $D_{q,m,\bsalpha}$ above is now replaced by
\begin{equation} \label{eqhg}
\sum_{i=1}^s h_ig_i \equiv 0 \ ({\rm mod} \ f).
\end{equation}
Let $\theta \in \FF_{q^m}$ be a root of $f$. Then each $\alpha_i \in \FF_{q^m}$ in the definition of a Vandermonde $(t,m,s)$-net over $\FF_q$ can be written as 
$\alpha_i=f_i(\theta)$ with a unique
$f_i \in H_{q,m}^*$. Thus, we arrive at the condition $0=\sum_{i=1}^s h_i(\alpha_i)=\sum_{i=1}^s h_i(f_i(\theta))$ in the definition of $D_{q,m,\bsalpha}$, which
is equivalent to
$$
\sum_{i=1}^s h_i \circ f_i \equiv 0 \ ({\rm mod} \ f).
$$
This is similar to~\eqref{eqhg}, but with the products $h_ig_i$ replaced by the compositions $h_i \circ f_i$. We note that polynomial lattice point sets belong to
the family of hyperplane nets (see \cite[Theorem~11.7]{DP10}), and so Vandermonde nets are in general not polynomial lattice point sets (see Remark~\ref{rehy}). }
\end{remark}

\begin{remark} \label{reev} {\rm
Since polynomial lattice point sets are available also for a reducible modulus $f \in \FF_q[x]$ (see \cite[Definition~10.1]{DP10}), we may extend the definition
of Vandermonde nets in an analogous way. For an arbitrary (and thus not necessarily irreducible) $f \in \FF_q[x]$ with $\deg(f)=m \ge 1$, we consider the
residue class ring $\FF_q[x]/(f)$. Given a dimension $s \ge 1$, we choose $g_1,\ldots,g_s \in H_{q,m}^*$. Note that $\FF_q[x]/(f)$ is a vector space over $\FF_q$,
with the canonical ordered basis $B$ given by the residue classes of the monomials $1,x,\ldots,x^{m-1}$ modulo $f$. Now we construct a digital net over $\FF_q$
with generating matrices $C^{(1)},\ldots,C^{(s)} \in \FF_q^{m \times m}$ as follows. For $1 \le j \le m$, the $j$th row vector of $C^{(1)}$ is given by the
coordinate vector of the residue class of $g_1^{j-1}$ modulo $f$ relative to the ordered basis $B$. If $s \ge 2$, then for $2 \le i \le s$ and $1 \le j \le m$, 
the $j$th row vector of $C^{(i)}$ is given by the coordinate vector of the residue class of $g_i^j$ modulo $f$ relative to the ordered basis $B$. We leave the
theory of these more general Vandermonde nets for future work. As for polynomial lattice point sets, the theory of general Vandermonde nets will be significantly
more complicated for reducible moduli $f$. }
\end{remark}

Now we establish existence results for Vandermonde $(t,m,s)$-nets over $\FF_q$ with a small quality parameter $t$. We use an elimination method which is inspired
by a similar method for polynomial lattice point sets (see~\cite[Section~3]{LLNS} and \cite[Section~10.1]{DP10}). We first show a simple enumeration result. 

\begin{lemma} \label{lepo}
For a prime power $q$, for $l\in\NN$ and $n \in\ZZ$, the number $A_q(l,n)$ of $(h_1,\ldots,h_l)\in\FF_q[x]^l$ with $h_i\neq 0$ and $h_i(0)=0$ 
for $1\leq i\leq l$ and $\sum_{i=1}^l\deg(h_i)= n$ is given by 
$$A_q(l,n)=\binom{n-1}{n-l}(q-1)^lq^{n-l},$$   
where we use the convention for binomial coefficients that $\binom{m}{k}=0$ whenever $k>m$ or $k<0$. 
\end{lemma}

\begin{proof}
Note that $h_i\neq 0$ and $h_i(0)=0$ imply $\deg(h_i)\geq 1$, and so trivially $A_q(l,n)=0$ for $n < l$.  
For $n \ge l$, we count the number of $l$-tuples $(d_1,\ldots,d_l)\in\NN^l$ such that $\sum_{i=1}^l d_i=n$, or equivalently the number of 
$l$-tuples $(d_1-1,\ldots,d_l-1)\in\NN_0^l$ such that $\sum_{i=1}^l (d_i-1)=n-l$. The latter number of $l$-tuples is given by $\binom{n-1}{n-l}$. 
For each $(d_1,\ldots,d_l) \in \NN^l$ with $\sum_{i=1}^l d_i=n$, there are $(q-1)^lq^{d_1-1}\cdots q^{d_l-1}=(q-1)^lq^{n-l}$ different 
$(h_1,\ldots,h_l)\in\FF_q[x]^l$ satisfying $h_i(0)=0$ and $\deg(h_i)=d_i$ for $1\leq i\leq l$, and the result follows.  
\end{proof}

Next we estimate the number $M_q(m,s,\sigma)$ of $(\alpha_1,\ldots,\alpha_s) \in \FF^s_{q^m}$ such that $\sum_{i=1}^s h_i(\alpha_i)= 0$ for at least 
one nonzero $s$-tuple $(h_1,\ldots,h_s) \in H_{q,m}^* \times H_{q,m}^{s-1}$ satisfying 
\begin{equation} \label{eqsi}
\deg^*(h_1)+\sum_{i=2}^s \deg(h_i) \leq \sigma.  
\end{equation}
We assume that $\sigma \in \ZZ$ and $0 \le \sigma \le m-1$. We have
\begin{equation} \label{eqmq}
M_q(m,s,\sigma) \le M_q^{(1)}(m,s,\sigma)+M_q^{(2)}(m,s,\sigma),
\end{equation}
where $M_q^{(1)}(m,s,\sigma)$, respectively $M_q^{(2)}(m,s,\sigma)$, is the number of $(\alpha_1,\ldots,\alpha_s) \in \FF_{q^m}^s$ such that $\sum_{i=1}^s
h_i(\alpha_i)=0$ for at least one nonzero $s$-tuple $(h_1,\ldots,h_s) \in H_{q,m}^* \times H_{q,m}^{s-1}$ with $h_1=0$, respectively $h_1 \ne 0$, satisfying~\eqref{eqsi}. 

We first consider $M_q^{(1)}(m,s,\sigma)$. Initially, we fix the number $d$ of zero entries in a nonzero $s$-tuple $(0,h_2,\ldots,h_s) \in H_{q,m}^* \times H_{q,m}^{s-1}$. 
Note that $1 \le d \le s-1$ and that~\eqref{eqsi} yields
$$
s-d \le \sum_{i=2}^s \deg(h_i) =: n \le \sigma +1.
$$
There exists an index $j \in \{2,\ldots,s\}$ such that $1\leq \deg(h_j)\leq \lfloor n/(s-d)\rfloor$. Then 
for each of the $q^{m(s-1)}$ choices of $(\alpha_1,\ldots,\alpha_{j-1},\alpha_{j+1},\ldots,\alpha_s) \in \FF_{q^m}^{s-1}$, there are at most 
$\lfloor n/(s-d)\rfloor$ choices of $\alpha_j \in \FF_{q^m}$ such that 
\begin{equation} \label{eqhj}
h_j(\alpha_j)=-\sum_{{i=1}\atop{i\neq j}}^s h_i(\alpha_i).
\end{equation}
There are $\binom{s-1}{d-1}$ choices for the positions of the zero entries in $(0,h_2,\ldots,h_s)$, and for each such choice there are $A_q(s-d,n)$ choices for the
$s-d$ nonzero entries. Using Lemma~\ref{lepo}, we arrive at the bound
\begin{eqnarray*}
 \lefteqn{
M_q^{(1)}(m,s,\sigma)}\\
& \le&  \sum_{d=1}^{s-1} \binom{s-1}{d-1} \sum_{n=s-d}^{\sigma +1} \binom{n-1}{n-s+d} (q-1)^{s-d}q^{n-s+d} q^{m(s-1)}
\left\lfloor \frac{n}{s-d} \right\rfloor.
\end{eqnarray*}

The estimation of $M_q^{(2)}(m,s,\sigma)$ proceeds in a similar way. Let $d$ be the number of zero entries in an $s$-tuple 
$(h_1,\ldots,h_s) \in H_{q,m}^* \times H_{q,m}^{s-1}$ with $h_1 \ne 0$. Then $0 \le d \le s-1$ and
$$
s-d-1 \le \sum_{i=1}^s \deg(h_i) =: n \le \sigma.
$$
There exists an index $j \in \{1,\ldots,s\}$ with $h_j \ne 0$ and $\deg(h_j) \le \lfloor n/(s-d) \rfloor$. As above, each choice of $(\alpha_1,\ldots,\alpha_{j-1},
\alpha_{j+1},\ldots,\alpha_s) \in \FF_{q^m}^{s-1}$ leaves at most $\lfloor n/(s-d) \rfloor$ choices of $\alpha_j \in \FF_{q^m}$ satisfying~\eqref{eqhj}. Since
$h_1 \ne 0$, there are $\binom{s-1}{d}$ choices for the positions of the zero entries in $(h_1,\ldots,h_s)$, and for each such choice there are $A_q(s-d,n+1)$ choices
for the $s-d$ nonzero entries (replace $h_1(x)$ by $xh_1(x)$ in order to arrive at the counting problem in Lemma~\ref{lepo}). Using Lemma~\ref{lepo}, we obtain
\begin{eqnarray*}
\lefteqn{M_q^{(2)}(m,s,\sigma)}\\ & \le & \sum_{d=0}^{s-1} \binom{s-1}{d} \sum_{n=s-d-1}^{\sigma} \binom{n}{n+1-s+d} (q-1)^{s-d} q^{n+1-s+d} q^{m(s-1)} 
\left\lfloor \frac{n}{s-d} \right\rfloor \\
& \le & \sum_{d=0}^{s-1} \binom{s-1}{d} \sum_{n=s-d}^{\sigma +1} \binom{n-1}{n-s+d} (q-1)^{s-d} q^{n-s+d} q^{m(s-1)} \left\lfloor \frac{n}{s-d} \right\rfloor.
\end{eqnarray*}
Now we use~\eqref{eqmq} and $\binom{s-1}{d-1} + \binom{s-1}{d} = \binom{s}{d}$ for $0 \le d \le s-1$, and this yields
\begin{eqnarray*}
\lefteqn{M_q(m,s,\sigma)}\\ &\le  & \sum_{d=0}^{s-1} \binom{s}{d} \sum_{n=s-d}^{\sigma +1} \binom{n-1}{n-s+d} (q-1)^{s-d} q^{n-s+d} q^{m(s-1)}
\left\lfloor \frac{n}{s-d} \right\rfloor\\
& = & q^{m(s-1)} \sum_{d=0}^{s-1} \binom{s}{d} (q-1)^{s-d}
\sum_{n=0}^{\sigma -s+d+1} \binom{n+s-d-1}{n} \left\lfloor \frac{n +s-d}{s-d} \right\rfloor q^n.
\end{eqnarray*}
We define 
$$
\Delta_q(s,\sigma):=\sum_{d=0}^{s-1} \binom{s}{d} (q-1)^{s-d} \sum_{n=0}^{\sigma -s+d} \binom{n+s-d-1}{n}
\left\lfloor \frac{n+s-d}{s-d} \right\rfloor q^n.
$$

Now we come to the crucial step: if $\Delta_q(s,\sigma +1)<q^m$ and therefore $M_q(m,s,\sigma)<q^{ms}$, then it follows that there exists at least one 
$\bsalpha=(\alpha_1,\ldots,\alpha_s)\in\FF^s_{q^m}$ such that 
$\sum_{i=1}^sh_i(\alpha_i)\neq0$ for every nonzero $s$-tuple $(h_1,\ldots,h_s)\in H_{q,m}^* \times H_{q,m}^{s-1}$ satisfying~\eqref{eqsi}. Hence for such $\bsalpha$, 
the figure of merit $\varrho(\bsalpha)$ satisfies $\varrho(\bsalpha) \geq \sigma+1$. From Theorem \ref{thqp} we deduce that the corresponding Vandermonde net 
satisfies $t\leq m-\sigma -1$. We have thus shown the following theorem. 

\begin{theorem} \label{thex}
Let $q$ be a prime power and let $s,m \in \NN$. If $\Delta_q(s,\sigma)<q^m$ for some $\sigma \in \NN$ with $\sigma \le m$, then there exists an 
$\bsalpha\in\FF^s_{q^{m}}$ with $\varrho(\bsalpha) \ge \sigma$. This $\bsalpha$ generates a Vandermonde 
$(t,m,s)$-net over $\FF_q$ with $t\leq m-\sigma$. 
\end{theorem}

\begin{coro} \label{coex}
Let $q$ be a prime power and let $s,m \in \NN$. Then there exists an $\bsalpha \in \FF_{q^m}^s$ with 
$$\varrho(\bsalpha)\geq \lfloor m-s\log_q m -3 \rfloor,$$
where $\log_q$ denotes the logarithm to the base $q$.
\end{coro}

\begin{proof}
For $s=1$ we can achieve $\varrho(\bsalpha)=\varrho((\alpha_1))=m$ by choosing $\alpha_1 \in \FF_{q^m}$ as a root of an irreducible polynomial over $\FF_q$
of degree $m$. If $s \ge 2$, it suffices to prove by Theorem~\ref{thex} that for 
$$\sigma_1 := \lfloor m-s\log_q m -3 \rfloor$$
we have $\Delta_q(s,\sigma_1)< q^m$. We can assume that $\sigma_1 \ge 1$, for otherwise the result is trivial. In the following, we derive a general 
upper bound on $\Delta_q(s,\sigma)$ for $\sigma \ge 1$ and then in a second step we use the specific form of $\sigma_1$. First of all, we have 
\begin{eqnarray*}
\Delta_q(s,\sigma) & \le & \sum_{d=0}^{s-1}\binom{s}{d}\left\lfloor\frac{\sigma}{s-d}\right\rfloor(q-1)^{s-d} \sum_{n=0}^{\sigma-s+d} 
\binom{n+s-d-1}{s-d-1} q^n\\
& \le & \sum_{d=0}^{s-1}\binom{s}{d} \frac{\sigma}{s-d} (q-1)^{s-d-1}\binom{\sigma-1}{s-d-1}q^{\sigma-s+d+1}\\
& \le & q^{\sigma-s+1}\sum_{d=0}^{s-1}\binom{s}{d} \frac{\sigma}{s-d} (q-1)^{s-d-1}\frac{(\sigma-1)^{s-d-1}}{(s-d-1)!}q^d\\
& = & s \sigma q^{\sigma-s+1} \sum_{d=0}^{s-1} \binom{s-1}{d} \frac{1}{(s-d) \cdot (s-d)!} [(q-1)(\sigma -1)]^{s-d-1} q^d.
\end{eqnarray*}
Now $(k+1) \cdot (k+1)! \ge 4^k$ for $k \ge 0$, and so $(s-d) \cdot (s-d)! \ge 4^{s-d-1}$ for $d=0,1,\ldots,s-1$. It follows that  
\begin{eqnarray*}
\Delta_q(s,\sigma)& \le & s \sigma q^{\sigma -s+1} \sum_{d=0}^{s-1} \binom{s-1}{d} \left[\frac{q-1}{4} (\sigma -1)\right]^{s-d-1} q^d\\
& = & s \sigma q^{\sigma -s+1} \left(\frac{q-1}{4} (\sigma -1)+q\right)^{s-1}\\
& = & s \sigma q^{\sigma} \left(\frac{q-1}{4q} (\sigma -1)+1\right)^{s-1},
\end{eqnarray*}
and so
$$
\Delta_q(s,\sigma) \le s \sigma q^{\sigma} \left(\frac{\sigma +3}{4}\right)^{s-1} \qquad \mbox{for } \sigma \ge 1.
$$
Now we use $s \ge 2$ and the special form of $\sigma_1$ to obtain
$$
\Delta_q(s,\sigma_1) < \sigma_1 q^{\sigma_1} (\sigma_1 +3)^{s-1} < mq^m m^{-s} m^{s-1} = q^m,
$$
and this yields the desired result. 
\end{proof}

We recall the definition of the \emph{star discrepancy} $D_N^*$ of any $N$ points $\by_1,\ldots,\by_N \in [0,1)^s$, namely
$$
D_N^*=\sup_J \left|\frac{Z(J)}{N} - \lambda_s(J) \right|,
$$
where the supremum is extended over all subintervals $J$ of $[0,1)^s$ with one vertex at the origin, where $Z(J)$ is the number of integers $n$ with $1 \le n \le N$
and $\by_n \in J$, and where $\lambda_s$ denotes the $s$-dimensional Lebesgue measure. Point sets with small star discrepancy are crucial ingredients of quasi-Monte
Carlo methods for numerical integration (see \cite[Chapter~2]{DP10}).

Using the well-known star discrepancy bound for $(t,m,s)$-nets in base $q$ (see \cite[Theorem~4.10]{N92}) together with Theorem \ref{thqp} and Corollary \ref{coex},
we arrive at the following result.  

\begin{coro} \label{coed}
Let $q$ be a prime power and let $s,m \in \NN$. Then there exists an $\bsalpha\in\FF_{q^m}^s$ such that the star discrepancy 
of the corresponding Vandermonde net satisfies 
$$D^*_{N}=\bigO_{q,s}\left(N^{-1} (\log N)^{2s-1}\right),$$ where $N=q^m$. 
\end{coro}

\section{Further existence results for small discrepancy} \label{sec:disc}

Throughout this section, we assume that $q$ is a prime, that $\FF_q$ is identified with $\ZZ_q$, and that $\psi : \FF_q \to \ZZ_q$ is the identity map. 
Then we know from \cite[Theorem~5.34]{DP10} that the star 
discrepancy of a digital net generated by $C^{(1)},\ldots,C^{(s)} \in \FF_q^{m\times m}$ satisfies 
\begin{equation} \label{eqgd}
D^*_{q^m}\leq 1-\left(1-\frac{1}{q^m}\right)^s+R_q(C^{(1)},\ldots,C^{(s)}),
\end{equation}
where 
$$
R_q(C^{(1)},\ldots,C^{(s)}):=\sum_{(\bk_1,\ldots,\bk_s) \in F'} \rho_q^{(s)}(\bk_1,\ldots,\bk_s)
$$ 
with 
$$
F'=\left\{(\bk_1,\ldots,\bk_s): \bk_1C^{(1)}+\cdots+\bk_sC^{(s)}=\bszero\right\}\setminus\{\bszero\}.
$$
Here $\bk_i \in \FF_q^m$ for $1 \le i \le s$. 
Furthermore $\rho_q^{(s)}(\bk_1,\ldots,\bk_s):=\prod_{i=1}^s \rho_q(\bk_i)$, where for $\bk =(k_1,\ldots,k_m) \in \FF_q^m$ we put 
$$
\rho_q(\bk):=\left\{
\begin{array}{ll}
1& \mbox{if } \bk=\bszero,\\
\frac{1}{q^{r}\sin(\pi k_r/q)}& \mbox{if } \bk =(k_1,\ldots,k_r,0,\ldots,0), \, k_r \ne 0.
\end{array}
\right.
$$

\begin{lemma} \label{lerq}
Let $C^{(1)},\ldots,C^{(s)} \in \FF_q^{m\times m}$ be the generating matrices of the Vandermonde net corresponding to 
$\bsalpha=(\alpha_1,\ldots,\alpha_s)\in\FF_{q^m}^s$. Then 
$$
R_q(\bsalpha):=R_q(C^{(1)},\ldots,C^{(s)})=\sum_{\bsh\in D'_{q,m,\bsalpha}} \rho_q^{(s)}(\bsh),
$$
where for $\bsh \in H_{q,m}^* \times H_{q,m}^{s-1}$ we put $\rho_q^{(s)}(\bsh)=\rho_q(xh_1(x)) \rho_q(h_2) \cdots \rho_q(h_s)$. Here for $h\in H_{q,m}$ we define 
$$
\rho_q(h)=\left\{
\begin{array}{ll}
1& \mbox{if } h=0,\\
\frac{1}{q^{r}\sin(\pi k_r/q)}& \mbox{if } h=k_1x+\cdots + k_rx^r, \, k_r \ne 0.
\end{array}
\right.
$$
\end{lemma}

\begin{proof}
This follows immediately from the form of the generating matrices $C^{(1)},\ldots,C^{(s)}$ of a Vandermonde net and from the definition of $D'_{q,m,\bsalpha}$.
\end{proof}

\begin{lemma}\label{lem:3}
For every prime $q$ and every $v,m\in\NN$, we have 
$$\sum_{h\in  H^{}_{q,m}}\rho_q(h)\leq \left\{
\begin{array}{ll}
\frac{m}{2}+1 & \mbox{ if } q=2,\\
\left(\frac{2}{\pi} \log q +\frac{2}{5}\right)m+1& \mbox{ if } q>2,
\end{array} \right.$$
and 
$$\sum_{\bsh\in H^*_{q,m}\times H^{v-1}_{q,m}}\rho^{(v)}_q(\bsh)\leq \left\{
\begin{array}{ll}
\left(\frac{m}{2}+1\right)^v & \mbox{ if } q=2,\\
\left(\left(\frac{2}{\pi} \log q +\frac{2}{5}\right)m+1\right)^v & \mbox{ if } q>2.
\end{array} \right.$$
\end{lemma}
\begin{proof}
This follows from the proof of \cite[Lemma~3.13]{N92}. 
\end{proof}

\begin{theorem} \label{thes}
Let $q$ be a prime and let $s,m \in \NN$. Then there exists an $\bsalpha\in\FF_{q^m}^s$ such that the star discrepancy of 
the corresponding Vandermonde net satisfies 
$$
D^*_{q^m} < 1-\left(1-\frac{1}{q^m}\right)^s +\frac{m}{q^m}\left\{
\begin{array}{ll}
\left(\frac{m}{2}+1\right)^s & \mbox{ if } q=2,\\
\left(\left(\frac{2}{\pi} \log q +\frac{2}{5}\right)m+1\right)^s & \mbox{ if } q>2.
\end{array} \right..
$$
\end{theorem}

\begin{proof}
We consider the average $M_{s,q,m}$ of $R_q(\bsalpha)$ over all $\bsalpha\in \FF_{q^m}^s$, that is, 
\begin{eqnarray*}
M_{s,q,m}&=&\frac{1}{q^{ms}}\sum_{\bsalpha\in \FF_{q^m}^s} R_q(\bsalpha)\\
&=& \frac{1}{q^{ms}} \sum_{\bsalpha\in \FF_{q^m}^s}\sum_{\bsh\in D'_{q,m,\bsalpha}} \rho_q^{(s)}(\bsh)\\
&=& \frac{1}{q^{ms}} \sum_{\bsh \in (H_{q,m}^* \times H_{q,m}^{s-1}) \setminus \{\bszero\}} A(\bsh) \rho_q^{(s)}(\bsh),
\end{eqnarray*}
where $A(\bsh)$ is the number of $\bsalpha =(\alpha_1,\ldots,\alpha_s) \in \FF_{q^m}^s$ such that $\sum_{i=1}^sh_i(\alpha_i)=0$. Now for every $\bsh \in
(H_{q,m}^* \times H_{q,m}^{s-1}) \setminus \{{\bf 0}\}$, $A(\bsh)$ is at most $mq^{m(s-1)}$. Hence
$$
M_{s,q,m} \le \frac{m}{q^{m}} \sum_{\bsh \in (H_{q,m}^* \times H_{q,m}^{s-1}) \setminus \{\bszero\}} \rho_q^{(s)}(\bsh) < \frac{m}{q^{m}} 
\sum_{\bsh \in H_{q,m}^* \times H_{q,m}^{s-1}} \rho_q^{(s)}(\bsh).
$$
The last sum can be bounded using Lemma \ref{lem:3}. The result of the theorem follows now from~\eqref{eqgd}. 
\end{proof}

In terms of the number $N=q^m$ of points, the bound on the star discrepancy $D_N^*$ in Theorem~\ref{thes} is of the form $D_N^*=O_s\left(N^{-1} (\log N)^{s+1}\right)$.

\section{An explicit construction} \label{seex}

In this section, $q$ is again an arbitrary prime power. For any dimension $s$ with $1 \le s \le q+1$ and any integer $m \ge 2$, we construct a Vandermonde $(t,m,s)$-net
over $\FF_q$ with the least possible quality parameter $t=0$. It is well known (see \cite[Corollary~4.21]{N92}) that for $m \ge 2$, a $(0,m,s)$-net in base $q$ 
cannot exist for $s \ge q+2$, and so our construction is best possible in terms of the dimension $s$.

Let $\theta \in \FF_{q^m}$ be a root of an irreducible polynomial over $\FF_q$ of degree $m \ge 2$. In the construction of Vandermonde nets in Section~\ref{sein},
we put $\alpha_1=\theta$ and (if $s \ge 2$) $\alpha_i=(\theta +c_i)^{-1}$ for $i=2,\ldots,s$, where $c_2,\ldots,c_s$ are distinct elements of $\FF_q$. Note that
$\theta +c_i \ne 0$ for $2 \le i \le s$ since $\theta \notin \FF_q$. Furthermore, the condition $s \le q+1$ guarantees that we can find $s-1$ distinct elements
$c_2,\ldots,c_s \in \FF_q$.

\begin{theorem} \label{thec}
Let $q$ be a prime power and let $s,m \in \NN$ with $s \le q+1$ and $m \ge 2$. Then the construction above yields a Vandermonde $(t,m,s)$-net over $\FF_q$ with $t=0$.
\end{theorem}

\begin{proof}
We proceed by Lemma~\ref{leed}. The case $s=1$ is trivial by the definition of $\theta$, and so we can assume that $s \ge 2$. For any integers $d_1,\ldots,d_s \ge 0$
with $\sum_{i=1}^s d_i=m$, we show that the $m$ elements $\theta^j$, $0 \le j \le d_1-1$, and $(\theta +c_i)^{-j}$ for $1 \le j \le d_i$ and $2 \le i \le s$ are
linearly independent over $\FF_q$. Consider a linear dependence relation
$$
\sum_{j=0}^{d_1-1} e_{1j} \theta^j + \sum_{i=2}^s \sum_{j=1}^{d_i} e_{ij} (\theta +c_i)^{-j} =0
$$
with all $e_{1j},e_{ij} \in \FF_q$. Multiply by $\prod_{k=2}^s (\theta +c_k)^{d_k}$ and put
$$
p_1(x)=\sum_{j=0}^{d_1-1} e_{1j} x^j \in \FF_q[x] \ \ \mbox{and} \ \ p_i(x)=\sum_{j=1}^{d_i} e_{ij} (x+c_i)^{d_i-j} \in \FF_q[x]
$$
for $2 \le i \le s$. Then
\begin{equation} \label{eqld}
p_1(\theta) \prod_{k=2}^s (\theta +c_k)^{d_k} + \sum_{i=2}^s p_i(\theta) \prod_{k=2 \atop k \ne i}^s (\theta +c_k)^{d_k} =0.
\end{equation}
Assume that for some integer $r$ with $2 \le r \le s$ we have $p_r(x) \ne 0$. Then $d_r \ge 1$ and $\deg(p_r(x)) < d_r$. On the left-hand side of~\eqref{eqld}
we have a polynomial in $\theta$ of degree $< \sum_{i=1}^s d_i=m$, and so this polynomial is the zero polynomial. Thus, we get the polynomial identity
\begin{equation} \label{eqpd}
p_1(x) \prod_{k=2}^s (x+c_k)^{d_k} + \sum_{i=2}^s p_i(x) \prod_{k=2 \atop k \ne i}^s (x+c_k)^{d_k} =0
\end{equation}
in $\FF_q[x]$. By considering this identity modulo $(x+c_r)^{d_r}$, we obtain
$$
p_r(x) \prod_{k=2 \atop k \ne r}^s (x+c_k)^{d_k} \equiv 0 \ ({\rm mod} \ (x+c_r)^{d_r}).
$$
The product over $k$ on the left-hand side is coprime to the modulus, and so it follows that $(x+c_r)^{d_r}$ divides $p_r(x)$. But $\deg(p_r(x)) < d_r$, and so
we arrive at a contradiction. Therefore $p_i(x)=0$ for $2 \le i \le s$, and so~\eqref{eqpd} shows that $p_1(x)=0$. Hence all coefficients $e_{1j},e_{ij} \in \FF_q$
in the original linear dependence relation are equal to $0$.
\end{proof}

The fact that we can explicitly construct optimal Vandermonde $(t,m,s)$-nets over $\FF_q$ for all dimensions $s \le q+1$ represents an advantage over polynomial 
lattice point sets (see Remark~\ref{repl} for the latter point sets). Explicit constructions of good polynomial lattice point sets are known only for $s=1$ and $s=2$ 
(see \cite[p.~305]{DP10}), whereas for $s \ge 3$ one has to resort to search algorithms in order to obtain good $s$-dimensional polynomial lattice point sets.

\section{Component-by-component constructions} \label{secb}

As in Section \ref{sec:disc} we assume that $q$ is a prime, that $\FF_q$ is identified with $\ZZ_q$, and that $\psi : \FF_q \to \ZZ_q$ is the identity map. 
Therefore the discrepancy bound in \eqref{eqgd} as well as Lemmas~\ref{lerq} and~\ref{lem:3} are valid. 
In the following, we introduce two component-by-component search algorithms for good Vandermonde nets in arbitrarily high dimensions, in the spirit of the 
search algorithms introduced in~\cite{Ko} and~\cite{SloRez02} for good lattice point sets and in~\cite{DKPS05} for good polynomial lattice point sets. 

\begin{algorithm}\label{algo:1}{\rm Given a prime $q$ and $s,m\in\NN$. \\
1. Choose $\alpha_1\in\FF_{q^m}$ as a root of an irreducible polynomial over $\FF_q$ of degree $m$.\\
2. For $d\in\NN$ with $2\leq d\leq s$, assume that we have already constructed $\alpha_1,\ldots,\alpha_{d-1}\in\FF_{q^m}$. We find $\alpha_d\in\FF_{q^m}$ that 
minimizes the quantity $R_q((\alpha_1,\ldots,\alpha_{d-1},\alpha_d))$ as a function of $\alpha_d$. }
\end{algorithm} 

\begin{theorem}\label{thm:4}
Let $q$ be a prime and let $s,m\in\NN$. Suppose that $\bsalpha=(\alpha_1,\ldots,\alpha_s)$ is constructed according to Algorithm \ref{algo:1}. Then for all 
$d\in\NN$ with $1\leq d\leq s$ we have $$R_q((\alpha_1,\ldots,\alpha_d))\leq \frac{m}{q^m}\left\{ 
\begin{array}{ll}
\left(\frac{m}{2}+1\right)^d & \mbox{ if } q=2,\\
\left(\left(\frac{2}{\pi} \log q +\frac{2}{5}\right)m+1\right)^d & \mbox{ if } q>2.
\end{array}
\right.$$
\end{theorem} 

\begin{proof}The proof is carried out by induction on $d$. For $d=1$ we have 
$$R_q((\alpha_1))=\sum_{h\in D'_{q,m,(\alpha_1)}}\rho_q(h)=0, $$
since $D'_{q,m,(\alpha_1)}$ is an empty set (note that $\alpha_1\in\FF_{q^m}$ is a root of an irreducible polynomial over $\FF_q$ of degree $m$ and therefore 
not a root of a nonzero polynomial $h\in H^*_{q,m}$). 

Suppose now that for some $1\leq d <s$, we have already constructed $(\alpha_1,\ldots,\alpha_d)\in\FF^d_{q^m}$ and the bounds in the theorem hold.
Then consider $(\alpha_1,\ldots,\alpha_d,\alpha_{d+1})$. We have 
\begin{eqnarray*}
R_q((\alpha_1,\ldots,\alpha_d,\alpha_{d+1}))&= & \sum_{(\bsh,h_{d+1})\in D'_{q,m,(\alpha_1,\ldots,\alpha_d,\alpha_{d+1})}}\rho_q^{(d)}(\bsh)\rho_q(h_{d+1})\\ 
&=& \sum_{\bsh\in D'_{q,m,(\alpha_1,\ldots,\alpha_d)}}\rho_q^{(d)}(\bsh) +\theta(\alpha_{d+1})\\
&=& R_q((\alpha_1,\ldots,\alpha_d)) +\theta(\alpha_{d+1}),
\end{eqnarray*}
where we split off the terms with $h_{d+1}=0$ and where
$$\theta(\alpha_{d+1})=\sum_{h_{d+1}\in {H_{q,m}}\setminus\{0\}} \rho_q(h_{d+1})
\sum_{{\bsh\in H^*_{q,m}\times H^{d-1}_{q,m}}\atop{(\bsh,h_{d+1})\in D'_{q,m,(\alpha_1,\ldots,\alpha_d,\alpha_{d+1})}}}\rho^{(d)}_q(\bsh).$$
Note that $\alpha_{d+1}$ is a minimizer of $R_q((\alpha_1,\ldots,\alpha_d,\cdot))$
and the only dependence on $\alpha_{d+1}$ is in $\theta$. Therefore $\alpha_{d+1}$ is a minimizer of $\theta$. We obtain 
\begin{eqnarray*}
\lefteqn{\theta(\alpha_{d+1})}\\
&\leq & \frac{1}{q^m}\sum_{\beta\in\FF_{q^m}}\theta(\beta)\\
&=& \frac{1}{q^m}\sum_{\beta\in\FF_{q^m}}\sum_{h_{d+1}\in {H_{q,m}}\setminus\{0\}} \rho_q(h_{d+1})
\sum_{{\bsh\in H^*_{q,m}\times H^{d-1}_{q,m}}\atop{(\bsh,h_{d+1})\in D'_{q,m,(\alpha_1,\ldots,\alpha_d,\beta)}}}\rho^{(d)}_q(\bsh)\\
&=& \frac{1}{q^m}\sum_{h_{d+1}\in {H_{q,m}}\setminus\{0\}} \rho_q(h_{d+1})\sum_{{\bsh\in H^*_{q,m}\times H^{d-1}_{q,m}}}\rho^{(d)}_q(\bsh)
\sum_{{\beta\in\FF_{q^m}}\atop{(\bsh,h_{d+1})\in D'_{q,m,(\alpha_1,\ldots,\alpha_d,\beta)}}}1.
\end{eqnarray*}
The condition $(\bsh,h_{d+1})\in D'_{q,m,(\alpha_1,\ldots,\alpha_d,\beta)}$ is equivalent to the equation $$h_{d+1}(\beta)=-\sum_{i=1}^dh_i(\alpha_i).$$
Since $h_{d+1}\in {H_{q,m}}\setminus\{0\}$, this equation has at most $m$ different solutions $\beta\in\FF_{q^m}$. 
Altogether we arrive at the bound
\begin{eqnarray*}
R_q((\alpha_1,\ldots,\alpha_d,\alpha_{d+1}))&\leq &  R_q((\alpha_1,\ldots,\alpha_d))\\
& & +\frac{m}{q^m}\sum_{h_{d+1}\in {H_{q,m}}\setminus\{0\}} \rho_q(h_{d+1})\sum_{{\bsh\in H^*_{q,m}\times H^{d-1}_{q,m}}}\rho^{(d)}_q(\bsh).
\end{eqnarray*}
The proof is completed by using the induction hypothesis and Lemma~\ref{lem:3}. 
\end{proof}

Theorem~\ref{thm:4} ensures that Algorithm~\ref{algo:1} produces vectors $\bsalpha\in\FF_{q^m}^s$ whose existence was guaranteed by Theorem~\ref{thes} in 
Section~\ref{sec:disc}. But Algorithm~\ref{algo:1} does not make use of the explicit construction in Section~\ref{seex} for low dimensions. The following 
algorithm suggests as initial values the explicitly constructed $\alpha_1,\ldots,\alpha_{q+1}$ of Section~\ref{seex} for a component-by-component procedure.

\begin{algorithm}\label{algo:2} {\rm Given a prime $q$ and $s,m\in\NN$ with $s>q+1$ and $m\geq 2$. \\
1. Choose $\alpha_1,\ldots,\alpha_{q+1}\in\FF_{q^m}$ as in the explicit construction of Section~\ref{seex}.\\
2. For $d\in\NN$ with $q+2\leq d\leq s$, assume that we have already constructed $\alpha_1,\ldots,\alpha_{d-1}\in\FF_{q^m}$. We find $\alpha_d\in\FF_{q^m}$ 
that minimizes the quantity $R_q((\alpha_1,\ldots,\alpha_{d-1},\alpha_d))$ as a function of $\alpha_d$. }
\end{algorithm} 

Although Algorithm~\ref{algo:2} starts from an (in the quality parameter point of view) optimal vector in $\FF_{q^m}^{q+1}$, one cannot be certain that 
Algorithm~\ref{algo:2} is competitive with  Algorithm~\ref{algo:1}. A straightforward generalization of the proof of Theorem~\ref{thm:4} would involve an 
upper bound for $R_q(C^{(1)},\ldots,C^{(q+1)})$, where $C^{(1)},\ldots,C^{(q+1)}$ are the generating matrices of a $(0,m,q+1)$-net over $\FF_q$. However,
the known bound for $R_q(C^{(1)},\ldots,C^{(q+1)})$ in \cite[Theorem~4.34]{N92} is not strong enough for all settings. Unfortunately, particularly for large 
values of $q$, one will obtain a weaker bound than in Theorem~\ref{thm:4}. It will be an interesting project for the future to implement Algorithms~\ref{algo:1}
and~\ref{algo:2} and compare their performance.

\subsection*{Acknowledgments}
R. Hofer was partially supported by the Austrian Science Foundation (FWF), Projects P21943 and P24302. Furthermore, she would like to thank Peter Hellekalek 
and the University of Salzburg, Austria, for their great hospitality during the research visit where parts of this paper were written.

\vspace{1cm}
\noindent{Roswitha Hofer, Institute of Financial Mathematics, University of Linz, Altenbergerstr.
69, A-4040 Linz, Austria; email: roswitha.hofer@jku.at} \\

\noindent{Harald Niederreiter, Johann Radon Institute for Computational and Applied
Mathematics, Austrian Academy of Sciences, Altenbergerstr. 69, A-
4040 Linz, Austria, and Department of Mathematics, University of Salzburg,
Hellbrunnerstr. 34, A-5020 Salzburg, Austria; email: ghnied@gmail.com}

\end{document}